\documentclass[11pt,a4paper]{article}

\usepackage[headinclude]{typearea}
\usepackage[english]{babel}
\usepackage[T1]{fontenc}
\usepackage{scrtime}
\usepackage{amsmath,amsthm,amsfonts,amssymb}                         \usepackage{xcolor}   
\usepackage {multicol,multirow} 
\usepackage[numbers]{natbib}   
\usepackage [scaled=.90]{helvet} 
\usepackage {courier}
\usepackage{enumerate} 
\usepackage{marvosym}
\usepackage{mathtools}
\mathtoolsset{showonlyrefs}
\usepackage{verbatim}
\usepackage{dsfont}
\usepackage{hyperref}
\makeatletter
\def\namedlabel#1#2{\begingroup
    #2%
    \def\@currentlabel{#2}%
    \phantomsection\label{#1}\endgroup
}

\newtheorem{theorem}{Theorem}[section]
\newtheorem{lemma}[theorem]{Lemma}
\newtheorem{corollary}[theorem]{Corollary}
\newtheorem{conjecture}[theorem]{Conjecture}

\theoremstyle{definition}
\newtheorem{assumption}{Assumption}[section]

\newtheorem{remark}[theorem]{Remark}

\newcommand{\E}{{\mathbb{E}}}
\newcommand{\N}{{\mathbb{N}}}
\renewcommand{\P}{{\mathbb{P}}}

\newcommand{\R}{{\mathbb{R}}}

\newcommand{\diff}{\mathop{}\!\mathrm{d}}
\newcommand{\cF}{{\cal F}}
\newcommand{\D}{{\mathbb{D}}}

\title{Lower error bounds and optimality of approximation for jump-diffusion SDEs with discontinuous drift}

\author{Pawe{\l} Przyby{\l}owicz \and Verena Schwarz \and Michaela Sz\"olgyenyi}
\date{Preprint, March 2023}

\begin{document}

\maketitle


\begin{abstract}
In this note we prove sharp lower error bounds for numerical methods for jump-diffusion stochastic differential equations (SDEs) with discontinuous drift. We study the approximation of jump-diffusion SDEs with non-adaptive as well as jump-adapted approximation schemes and provide lower error bounds of order $3/4$ for both classes of approximation schemes. This yields optimality of the transformation-based jump-adapted quasi-Milstein scheme.\\

\end{abstract}


\section{Introduction}\label{sec:intro}

Our aim is to provide lower error bounds as well as an optimality result for the approximation of SDEs with discontinuous drift.

For this, we consider the following time-homogeneous jump-diffusion stochastic differential equation (SDE) with additive noise
\begin{align}\label{eq:SDE}
\diff X_t &= \mu(X_t) \diff t + \diff W_t+ \diff N_t , \quad t\in[0,1], \quad X_0= \xi.
\end{align}
Here $\xi \in\R$, $\mu \colon\R\to \R$ is a measurable function, $W=(W_t)_{t\in[0,1]}$ is a Brownian motion, and $(N_t)_{t\in[0,1]}$ is a homogeneous Poisson process with intensity $\lambda\in  (0,\infty)$ both defined on a filtered probability space $(\Omega,\cF,\mathbb{F},\P)$, where $\mathbb{F}=(\mathbb{F}_t)_{t\in[0,1]}$ is a filtration satisfying the usual conditions.

By \cite{PS20,PSX21} existence and uniqueness of the solution of the above SDE is well settled under our assumptions.

In the case without jumps upper error bounds for SDEs with discontinuous drift can be found in
\cite{halidias2006,halidias2008,LS16,ngo2016,LS17,ngo2017a,ngo2017b,LS18,muellergronbach2019,muellergronbach2019b,NSS19,Dareiotis2020,Neuenkirch2021,Dareiotis2021,Yaroslavtseva2022,Mullergronbach2022,Hu2022,Spendier2022} and upper error bounds in the jump-diffusion case are given in \cite{PS20,PSS22}.
For lower bounds in more regular settings we refer to \cite{Clark1980,Hofmann2001,Muellergronbach2004};
for lower bounds and optimality results for the approximation of jump-diffusion SDEs with continuous drift, see \cite{Przybylowicz2016,Kaluza2018,Przybylowicz2019a,Przybylowicz2019b,Przybylowicz2022,Kaluza2022}.
The first results on lower bounds for jump-free SDEs with discontinuous drift are \cite{hefter2018,MY20, Ellinger2022}. In \cite{MY20} convergence order $3/4$ is proven to be optimal for non-adaptive approximation schemes based on a finite number of approximations of the driving Brownian motion under the additional assumptions that the drift coefficient is monotonically increasing and bounded. In \cite{Ellinger2022} these additional assumptions are dropped by using the transformation technique introduced in \cite{LS16, LS17, muellergronbach2019b}.

Our note adds to the literature in the following aspect:
We provide the first lower error bound for the approximation of SDEs with discontinuous drift and in the presence of jumps. We prove bounds for both non-adaptive as well as jump-adapted schemes of order $3/4$ in $L^1$. The rate $3/4$ turns out to be optimal in the case of jump-adapted approximation schemes, since in \cite{PSS22} an upper error bound of the same order is proven in a more general setting for the transformation-based jump-adapted quasi-Milstein scheme.

\section{Preliminaries}\label{sec:pre}

For a Lipschitz continuous function \(f\) we denote by \(L_f\) its Lipschitz constant. For a vector $v\in\R^d$ with $d\in\N$ we denote by $v^T$ its transpose. For an integer $d\in\N$ and we denote by $\D([0,1],\R^d)$ the space of all càdlàg functions from $[0,1]$ to $\R^d$ equipped with the Skorokhod topology. Denote by $\nu_1, \ldots \nu_{N_1}$ the points of discontinuity of the path of the Poisson process on the interval $[0,1]$ ordered strictly increasing. Further, denote by $l^1$ the space of all real valued sequences $(x_n)_{n\in\N}$ with $\|(x_n)_{n\in\N}\| = \sum_{n=1}^\infty |x_n| <\infty$.

\begin{assumption}\label{ass:ex-un}
We assume that the drift coefficient $\mu\colon \R\to\R$ of SDE \eqref{eq:SDE} satisfies:
\begin{itemize}
\item[\namedlabel{mu1}{(i)}] it is piecewise Lipschitz continuous, i.e. there exist $k\in\N$ 
and 
$-\infty=\zeta_0<\zeta_1<\ldots < \zeta_k <\zeta_{k+1}=\infty$ such that
 $\mu$ is Lipschitz continuous on  $(\zeta_{i}, \zeta_{i+1})$ for every $i\in\{0, \ldots, k\}$.
\item[\namedlabel{mu2}{(ii)}] it is differentiable with Lipschitz continuous derivative on  $(\zeta_{i}, \zeta_{i+1})$  for every $i\in\{0, \ldots, k\}$.
\item[\namedlabel{mu3}{(iii)}] there exists \(i\in\{1,\dots,k\}\) such that \(\mu(\zeta_i+)\not=\mu(\zeta_i-)\).
\end{itemize}
\end{assumption}

Under the same assumptions on $\mu$ we introduce the jump-free SDE
\begin{equation}
    \begin{aligned}\label{newSDE}
    \diff Y_t = \mu(Y_t)\diff t + \diff W_t, \quad t\in[0,1], \quad Y_0 = \xi.
    \end{aligned}
\end{equation}
The existence and uniqueness of the solution $Y$ is guaranteed by \cite{LS16, LS17}.

For the proof of our main result we will reduce the complexity of our problem to be able to apply the lower bound for the jump-free case \cite[Theorem 1]{Ellinger2022}. For this we will need a joint functional representation for the solutions of SDE \eqref{eq:SDE} and SDE \eqref{newSDE}.

\begin{lemma}\label{FuncRep}
Assume that $\mu$ satisfies Assumption \ref{ass:ex-un} \ref{mu1}. Then there exists a Skorokod measurable mapping
\begin{equation*}
\begin{aligned}
F\colon \R \times \D([0,1],\R^3)\to \D([0,1],\R)
\end{aligned}
\end{equation*}
such that $F(\xi,(Id,W,N)^T)$ is the unique strong solution of SDE \eqref{eq:SDE} and $F(\xi,(Id,W,0)^T)$ is the unique strong solution of SDE \eqref{newSDE}.
\end{lemma}

\begin{proof}
First we recall the transformation form \cite{PS20}, which is a variant of the original construction from \cite{LS17}. 
Let $\phi\colon \R \to \R$ be defined by
\begin{align*}
\phi(u)=
\begin{cases}
(1+u)^3(1-u)^3 & \text{if } |u|\le 1,\\
0 & \text{otherwise}.
\end{cases}
\end{align*}
Based on this we define $G$ by 
\begin{align*}
G(x)=x+ \sum_{j=1}^k \alpha_j\phi\!\left(\frac{x-\zeta_j}{c}\right)(x-\zeta_j)|x-\zeta_j| ,
\end{align*}
where
\begin{align*}
\alpha_j
=\frac{\mu(\zeta_j-)-\mu(\zeta_j+)}{2},
\qquad j\in\{1,\dots,k\}
\end{align*}
and
\begin{align*} 
c\in\bigg(0,\min\bigg\{\min\limits_{1\le j\le k}\frac{1}{6|\alpha_j|},\min\limits_{1\le j\le k-1}\frac{\zeta_{j+1}-\zeta_j}{2}\bigg\}\bigg),
\end{align*}
when we use the convention $1/0=\infty$. This function $G$ satisfies the following properties, see \cite[Lemma 3.8 and Lemma 2.2]{LS17}:
\begin{itemize}
    \item  For all \(x\in\R\) it holds that $G'(x)>0$;
    \item \(G\) has a global inverse $G^{-1}\colon \R\to\R$;
    \item $G$ and $G^{-1}$ are Lipschitz continuous;
    \item $G\in C^1_b$, which means that $G$ is continuously differentiable and has a bounded derivative;
    \item $G'$ is a Lipschitz continuous function.
\end{itemize}

Now we apply the transformation $G$ to SDE \eqref{eq:SDE} and SDE \eqref{newSDE}. For SDE \eqref{eq:SDE} we get as in the proof of \cite[Theorem 3.1]{PS20} that $Z=G(X)$ satisfies 
\begin{align*}
	\diff Z_t=\widetilde\mu (Z_t)\diff t+\widetilde\sigma(Z_t)\diff W_t+\widetilde \rho(Z_{t-})\diff N_t,  \quad t\in[0,1], \quad Z_0= G(\xi),
\end{align*}
where for all \(\xi\in\R\)
\begin{align*}
\widetilde\mu =\Big(G'\mu + \frac{1}{2} G''\Big)\circ G^{-1}, \quad 
\widetilde\sigma=G'\circ G^{-1},\text{ and }\quad
\label{t_rho_def}
\widetilde \rho(x)=G(G^{-1}(x)+1)-x.
\end{align*}
It holds that $\widetilde\mu$, $\widetilde\sigma$, and  $\widetilde\rho$ are Lipschitz continuous.
We can also apply $G$ to SDE \eqref{newSDE} and obtain that $\widetilde Z = G(Y)$ satisfies 
\begin{align*}
    \diff \widetilde Z_t=\widetilde\mu (\widetilde Z_t)\diff t+\widetilde\sigma( \widetilde Z_t)\diff W_t,  \quad t\in[0,1], \quad \widetilde Z_0= G(\xi).
\end{align*} 
We rewrite these SDEs to
\begin{equation*}
\begin{aligned}
\diff Z_t &= (\widetilde\mu (Z_{t-}),\widetilde\sigma(Z_{t-}),\widetilde \rho(Z_{t-})) \diff (t,W_t,N_t)^T, \quad t\in[0,1], \quad Z_0= G(\xi),
\end{aligned}
\end{equation*}
and
\begin{equation*}
\begin{aligned}
\diff \widetilde Z_t &= (\widetilde\mu (\widetilde Z_{t-}),\widetilde\sigma( \widetilde Z_{t-}),\widetilde \rho(\widetilde Z_{t-})) \diff (t,W_t,0)^T, \quad t\in[0,1], \quad \widetilde Z_0= G(\xi).
\end{aligned}
\end{equation*}
The function $(\widetilde\mu,\widetilde\sigma,\widetilde\rho)$ satisfies \cite[Assumption 2.1]{PSSS22} and hence by the Skorohod measurable universal function representation proven in \cite[Theorem 3.1]{PSSS22} there exists a function $\Psi\colon \R\times \D([0,1],\R^3) \to \D([0,1],\R)$ such that
\begin{equation*}
\begin{aligned}
Z = \Psi(G(\xi),(t,W_t,N_t)^T_{t\in[0,1]}) \text{ and } \widetilde Z = \Psi(G(\xi),(t,W_t,0)^T_{t\in[0,1]}).
\end{aligned}
\end{equation*}
Next we apply for all $t\in[0,1]$, $G^{-1}$ to $Z_t$ respectively $\widetilde Z_t$ to obtain 
\begin{equation*}
    (X_t)_{t\in[0,1]} =\Biggl(G^{-1}\Bigl[\Bigl(\Psi(G(\xi),(t,W_t,N_t)^T_{t\in[0,1]})\Bigr)(t)\Bigr]\Biggr)_{t\in [0,1]}
\end{equation*}
and
\begin{equation*}
    (Y_t)_{t\in[0,1]} =\Biggl(G^{-1}\Bigl[\Bigl(\Psi(G(\xi),(t,W_t,0)^T_{t\in[0,1]})\Bigr)(t)\Bigr]\Biggr)_{t\in [0,1]}.
\end{equation*}

Defining
\begin{equation*}
\begin{aligned}
F\colon \R\times \D([0,1],\R^3) \to \D([0,1],\R), \quad F(\xi,x)(\cdot) = G^{-1}(\Psi(G(\xi),x)(\cdot)), 
\end{aligned}
\end{equation*}
which is as a concatenation of measurable functions again measurable, proves the claim. 
\end{proof}

\section{Main result}

\begin{theorem}\label{LBada}
 Assume that $\mu$ satisfies Assumption \ref{ass:ex-un}. Let $\xi\in\R$, let $W\colon[0,1]\times \Omega \to \R$ be a Brownian motion and $N\colon[0,1]\times \Omega \to \R$ a Poisson process with intensity $\lambda\in (0,\infty)$, which is independent of $W$. Let $X \colon [0,1]\times \Omega\to \R$ be the strong solutions of SDE \eqref{eq:SDE} on the time-interval $[0,1]$ with initial value $\xi$, driving Brownian motion $W$, and driving Poisson process $N$. Then  
 there exist $c\in (0,\infty)$ such that for all $n\in \N$,
 \begin{equation*}
\inf_{\substack{
       t_1,\dots ,t_n \in [0,1]\\
        g \colon\R^n \times l^1 \times \D([0,1],\R) \to \R\\
        \text{ measurable} \\
       }}	 
       \E\bigl[|X_1-g((W_{t_1}, \ldots, W_{t_n}),(W_{\nu_1},...,W_{\nu_{N_1}},0,...), (N_t)_{t\in[0,1]}))|\bigr]\geq \frac{c}{n^{3/4}}. 
\end{equation*}
\end{theorem}

\begin{proof}
Fix $n\in\N$, $t_1,\dots ,t_n \in [0,1]$, and $g \colon\R^n \times l^1 \times \D([0,1],\R) \to \R$ measurable. In the following we make use of considering only those $\omega\in\Omega$ for which $N_1(\omega)=0$. This implies that $N$ has no jump until time $1$. Using Lemma \ref{FuncRep} and observing that $(F({\xi}, (Id,W,0)^T)_t)_{t\in[0,1]}$ is the solution of \eqref{newSDE} we get
\begin{equation*}
    \begin{aligned}
    &\E\Bigl[\Bigl|X_1-g\Bigl((W_{t_1}, \ldots, W_{t_n}),(W_{\nu_1},...,W_{\nu_{N_1}},0,...), (N_t)_{t\in[0,1]}\Bigr)\Bigl|\Bigr]\\
    &\geq \E\bigl[|X_1-g((W_{t_1}, \ldots, W_{t_n}),(W_{\nu_1},...,W_{\nu_{N_1}},0,...), (N_t)_{t\in[0,1]})|\cdot\mathds{1}_{\{N_1=0\}}\bigr]\\
    &= \E\bigl[|F({\xi},(Id,W,N)^T)_1-g((W_{t_1}, \ldots, W_{t_n}),(W_{\nu_1},...,W_{\nu_{N_1}},0,...), (N_t)_{t\in[0,1]})|\cdot\mathds{1}_{\{N_1=0\}}\bigr]\\
    &= \E\bigl[|F({\xi},(Id,W,0)^T)_1-g((W_{t_1}, \ldots, W_{t_n}),(0,...), (0)_{t\in[0,1]})|\cdot\mathds{1}_{\{N_1=0\}}\bigr]\\
    &= \E\bigl[|F({\xi},(Id,W,0)^T)_1-g((W_{t_1}, \ldots, W_{t_n}),(0,...), (0)_{t\in[0,1]})|\bigr]\P(N_1=0)\\
    &= \E\bigl[|Y_1-g((W_{t_1}, \ldots, W_{t_n}),(0,...), (0)_{t\in[0,1]})|\bigr]\P(N_1=0).
    \end{aligned}
\end{equation*}
It holds that $\P(N_1=0) = \exp(-\lambda)>0$.
By  \cite[Theorem 1]{Ellinger2022} there exists a constant $c\in(0,\infty)$ such that for all $n\in\N$,
\begin{equation*}
    \begin{aligned}
    &\inf_{\substack{
    t_1,\dots ,t_n \in [0,1]\\
    h \colon \R^{n} \to \R \text{ measurable} \\
    }}	 \E\bigl[|Y_1-h(W_{t_1}, \ldots, W_{t_n})|\bigr]
    \geq \frac{c}{n^{3/4}}.
    \end{aligned}
\end{equation*}
Since $g((W_{t_1}, \ldots, W_{t_n}),(0,...), (0)_{t\in[0,1]}))$ can be also interpreted as $h(W_{t_1}, \ldots, W_{t_n})$ for $h\colon\R^n\to\R$, $h(x)=g(x, (0,..),(0)_{t\in [0,1]})$ for all $x\in\R^n$ it holds that
\begin{equation*}
    \begin{aligned}
    &\E\bigl[|X_1-g((W_{t_1}, \ldots, W_{t_n}),(W_{\nu_1},...,W_{\nu_{N_1}},0,...), (N_t)_{t\in[0,1]}))|\bigr]\\
    &\geq \E\bigl[|{Y}_1-g((W_{t_1}, \ldots, W_{t_n}),(0,...), (0)_{t\in[0,1]}))|\big|\bigr]\P(N_1=0)\\
    &\geq \inf_{\substack{
    t_1,\dots ,t_n \in [0,1]\\
    h \colon \R^{n} \to \R \text{ measurable} \\
    }}	 \E\bigl[|{Y}_1-h(W_{t_1}, \ldots, W_{t_n})|\bigr] \cdot \exp(-\lambda)\\
    &\geq \exp(-\lambda) \frac{c}{n^{3/4}}.
    \end{aligned}
\end{equation*}
Since this holds for all $n\in\N$, $t_1,\dots ,t_n \in [0,1]$, and $g \colon\R^n \times l^{1} \times \D([0,1],\R) \to \R$ measurable, the claim is proven.
\end{proof}

In \cite[Theorem 4.4]{PSS22} an upper error bound of order $3/4$ is proven for the so-called transformation-based jump-adapted quasi-Milstein scheme. In Theorem \ref{LBada} we prove a lower error bound of the same order, which holds for every class of SDEs that contain an SDE \eqref{eq:SDE} satisfying Assumption \ref{ass:ex-un}.
Hence, this implies optimality in the worst-case setting of the transformation-based jump-adapted quasi-Milstein scheme for every class of jump-diffusion SDEs that contain a particular equation \eqref{eq:SDE} and for which the coefficients satisfy the assumptions of \cite[Theorem 4.4]{PSS22}. In particular, we obtain that for SDE \eqref{eq:SDE} under Assumption \ref{ass:ex-un} the error rate $3/4$ is optimal and it is achieved by the transformation based jump-adapted quasi-Milstein scheme.

\begin{corollary}
    Assume that $\mu$ satisfies Assumption \eqref{ass:ex-un}. Let $\xi\in\R$, let $W\colon[0,1]\times \Omega \to \R$ be a Brownian motion and $N\colon[0,1]\times \Omega \to \R$ a Poisson process with intensity $\lambda\in (0,\infty)$, which is independent of $W$. Let $X \colon [0,1]\times \Omega\to \R$ be the strong solutions of SDE \eqref{eq:SDE} on the time-interval $[0,1]$ with initial value $\xi$, driving Brownian motion $W$, and driving Poisson process $N$. Further, let $X^{M,n}$ be the transformation based jump-adapted quasi-Milstein scheme based on $n$ equidistant gird points for all $n\in\N$. Then it holds that 
    \begin{equation*}
        \E\Bigl[\sup_{t\in[0,1]}\big|X_t-X^{M,n}_t\big|\Bigr] = \Theta(n^{-3/4}). 
    \end{equation*}    
\end{corollary}

The following corollary shows that a lower error bound of order $3/4$ can also be obtained for non-adaptive approximation schemes. 

\begin{corollary}\label{LBeq}
Assume that $\mu$ satisfies Assumption \ref{ass:ex-un}. Let $\xi\in\R$, let $W\colon[0,1]\times \Omega \to \R$ be a Brownian motion and $N\colon[0,1]\times \Omega \to \R$ a Poisson process with intensity $\lambda\in (0,\infty)$, which is independent of $W$. Let $X \colon [0,1]\times \Omega\to \R$ be the strong solutions of SDE \eqref{eq:SDE} on the time-interval $[0,1]$ with initial value $\xi$, driving Brownian motion $W$, and driving Poisson process $N$. Then  
 there exist $c\in (0,\infty)$ such that for all $n\in \N$,
 \begin{equation*}
\inf_{\substack{
       t_1,\dots ,t_n \in [0,1]\\
        g \colon\R^{2n} \to \R
        \text{ measurable} \\
       }}	 
       \E\bigl[|X_1-g(W_{t_1}, \ldots, W_{t_n},N_{t_1},\ldots N_{t_n})|\bigr]\geq \frac{c}{n^{3/4}}. 
\end{equation*}
\end{corollary}

\begin{remark} 
Corollary \ref{LBeq} does not imply any optimality results, since there is still a gap between the upper and lower bounds known from literature. The best upper bounds are provided in \cite{PS20}. There it is proven that the Euler--Maruyama scheme has convergence order $1/2$ in $L^2$, which implies the same rate in $L^1$. 
\end{remark}

\begin{conjecture}  
For SDE \eqref{eq:SDE} satisfying Assumptions \ref{ass:ex-un} the convergence rate $1/2$ is optimal in $L^2$ for non-adaptive algorithms, which evaluate $W$, $N$ at fixed time points.
\end{conjecture}


\section*{Acknowledgements}
V. Schwarz and M. Szölgyenyi are supported by the Austrian Science Fund (FWF): DOC 78.


\setlength{\bibsep}{0pt plus 0.3ex}


\vspace{0.2em}
\centerline{\underline{\hspace*{16cm}}}

\noindent Pawe{\l} Przyby{\l}owicz  \\
Faculty of Applied Mathematics, AGH University of Science and Technology, Al.~Mickiewicza 30, 30-059 Krakow, Poland\\
pprzybyl@agh.edu.pl\\

\noindent Verena Schwarz \Letter \\
Department of Statistics, University of Klagenfurt, Universit\"atsstra\ss{}e 65-67, 9020 Klagenfurt, Austria\\
verena.schwarz@aau.at\\

\noindent Michaela Sz\"olgyenyi \\
Department of Statistics, University of Klagenfurt, Universit\"atsstra\ss{}e 65-67, 9020 Klagenfurt, Austria\\
michaela.szoelgyenyi@aau.at\\



{\footnotesize
\begin{thebibliography}{34}
\providecommand{\natexlab}[1]{#1}
\providecommand{\url}[1]{\texttt{#1}}
\expandafter\ifx\csname urlstyle\endcsname\relax
  \providecommand{\doi}[1]{doi: #1}\else
  \providecommand{\doi}{doi: \begingroup \urlstyle{rm}\Url}\fi

\bibitem[Clark and Cameron(1980)]{Clark1980}
J.~M.~C. Clark and R.~J. Cameron.
\newblock The maximum rate of convergence of discrete approximations for stochastic differential equations, in: Stochastic differential systems filtering and control.
\newblock \emph{Stochastic Differential Systems Filtering and Control}, pages 162--171, 1980.

\bibitem[Dareiotis and Gerencs{\'e}r(2020)]{Dareiotis2020}
K.~Dareiotis and M.~Gerencs{\'e}r.
\newblock On the regularisation of the noise for the {E}uler--{M}aruyama scheme with irregular drift.
\newblock \emph{Electronic Journal of Probability}, 25:\penalty0 1--18, 2020.

\bibitem[Dareiotis et~al.(2023)Dareiotis, Gerencs{\'e}r, and L{\^e}]{Dareiotis2021}
K.~Dareiotis, M.~Gerencs{\'e}r, and K.~L{\^e}.
\newblock Quantifying a convergence theorem of {G}y{\"o}ngy and {K}rylov.
\newblock \emph{The Annals of Applied Probability}, 33\penalty0 (3):\penalty0 2291--2323, 2023.

\bibitem[Ellinger(2022)]{Ellinger2022}
S.~Ellinger.
\newblock Sharp lower error bounds for strong approximation of {SDE}s with piecewise {L}ipschitz continuous drift coefficient.
\newblock \emph{arXiv:2303.05346}, 2022.

\bibitem[Halidias and Kloeden(2006)]{halidias2006}
N.~Halidias and P.~E. Kloeden.
\newblock A note on strong solutions of stochastic differential equations with a discontinuous drift coefficient.
\newblock \emph{Journal of Applied Mathematics and Stochastic Analysis}, 2006:\penalty0 1--6, 2006.

\bibitem[Halidias and Kloeden(2008)]{halidias2008}
N.~Halidias and P.~E. Kloeden.
\newblock A note on the {E}uler–{M}aruyama scheme for stochastic differential equations with a discontinuous monotone drift coefficient.
\newblock \emph{BIT Numerical Mathematics}, 48\penalty0 (1):\penalty0 51--59, 2008.

\bibitem[Hefter et~al.(2018)Hefter, Herzwurm, and M\"uller-Gronbach]{hefter2018}
M.~Hefter, A.~Herzwurm, and T.~M\"uller-Gronbach.
\newblock Lower error bounds for strong approximation of scalar {SDE}s with non-{L}ipschitzian coefficients.
\newblock \emph{Annals of Applied Probability}, 2018.
\newblock Forthcoming.

\bibitem[Hofmann et~al.(2001)Hofmann, M{\"u}ller-Gronbach, and Ritter]{Hofmann2001}
N.~Hofmann, T.~M{\"u}ller-Gronbach, and K.~Ritter.
\newblock The optimal discretization of stochastic differential equations.
\newblock \emph{Journal of Complexity}, 17\penalty0 (1):\penalty0 117--153, 2001.

\bibitem[Hu and Gan(2022)]{Hu2022}
H.~Hu and S.~Gan.
\newblock Strong convergence of the tamed {E}uler scheme for scalar {SDE}s with superlinearly growing and discontinuous drift coefficient.
\newblock \emph{arXiv:2206.00088}, 2022.

\bibitem[Ka{\l}u{\.z}a(2022)]{Kaluza2022}
A.~Ka{\l}u{\.z}a.
\newblock Optimal global approximation of systems of jump-diffusion {SDE}s on equidistant mesh.
\newblock \emph{Applied Numerical Mathematics}, 179:\penalty0 1--26, 2022.

\bibitem[Ka{\l}u{\.z}a and Przyby{\l}owicz(2018)]{Kaluza2018}
A.~Ka{\l}u{\.z}a and P.~Przyby{\l}owicz.
\newblock Optimal global approximation of jump-diffusion {SDE}s via path-independent step-size control.
\newblock \emph{Applied Numerical Mathematics}, 128:\penalty0 24--42, 2018.

\bibitem[Leobacher and Sz\"olgyenyi(2016)]{LS16}
G.~Leobacher and M.~Sz\"olgyenyi.
\newblock A numerical method for sdes with discontinuous drift.
\newblock \emph{BIT Numerical Mathematics}, 56\penalty0 (1):\penalty0 151--162, 2016.

\bibitem[Leobacher and Sz\"olgyenyi(2017)]{LS17}
G.~Leobacher and M.~Sz\"olgyenyi.
\newblock A strong order 1/2 method for multidimensional sdes with discontinuous drift.
\newblock \emph{The Annals of Applied Probability}, 27\penalty0 (4):\penalty0 2383--2418, 2017.

\bibitem[Leobacher and Sz\"olgyenyi(2018)]{LS18}
G.~Leobacher and M.~Sz\"olgyenyi.
\newblock Convergence of the {E}uler-{M}aruyama method for multidimensional {SDE}s with discontinuous drift and degenerate diffusion coefficient.
\newblock \emph{Numerische Mathematik}, 138\penalty0 (1):\penalty0 219--239, 2018.

\bibitem[M{\"u}ller-Gronbach(2004)]{Muellergronbach2004}
T.~M{\"u}ller-Gronbach.
\newblock Optimal pointwise approximation of {SDE}s based on {B}rownian motion at discrete points.
\newblock \emph{The Annals of Applied Probability}, 14\penalty0 (4):\penalty0 1605--1642, 2004.

\bibitem[M{\"u}ller-Gronbach and Yaroslavtseva(2020)]{muellergronbach2019}
T.~M{\"u}ller-Gronbach and L.~Yaroslavtseva.
\newblock On the performance of the {E}uler--{M}aruyama scheme for {SDE}s with discontinuous drift coefficient.
\newblock In \emph{Annales de l’Institut Henri Poincar{\'e}-Probabilit{\'e}s et Statistiques}, volume~56, pages 1162--1178, 2020.

\bibitem[M{\"u}ller-Gronbach and Yaroslavtseva(2022)]{muellergronbach2019b}
T.~M{\"u}ller-Gronbach and L.~Yaroslavtseva.
\newblock A strong order 3/4 method for {SDE}s with discontinuous drift coefficient.
\newblock \emph{IMA Journal of Numerical Analysis}, 42\penalty0 (1):\penalty0 229--259, 2022.

\bibitem[M{\"u}ller-Gronbach and Yaroslavtseva(2023)]{MY20}
T.~M{\"u}ller-Gronbach and L.~Yaroslavtseva.
\newblock Sharp lower error bounds for strong approximation of {SDE}s with discontinuous drift coefficient by coupling of noise.
\newblock \emph{The Annals of Applied Probability}, 33\penalty0 (2):\penalty0 1102--1135, 2023.

\bibitem[M{\"u}ller-Gronbach et~al.(2022)M{\"u}ller-Gronbach, Sabanis, and Yaroslavtseva]{Mullergronbach2022}
T.~M{\"u}ller-Gronbach, S.~Sabanis, and L.~Yaroslavtseva.
\newblock Existence, uniqueness and approximation of solutions of {SDE}s with superlinear coefficients in the presence of discontinuities of the drift coefficient.
\newblock \emph{arXiv:2204.02343}, 2022.

\bibitem[Neuenkirch and Sz{\"o}lgyenyi(2021)]{Neuenkirch2021}
A.~Neuenkirch and M.~Sz{\"o}lgyenyi.
\newblock The {E}uler--{M}aruyama scheme for {SDE}s with irregular drift: convergence rates via reduction to a quadrature problem.
\newblock \emph{IMA Journal of Numerical Analysis}, 41\penalty0 (2):\penalty0 1164--1196, 2021.

\bibitem[Neuenkirch et~al.(2019)Neuenkirch, Sz\"olgyenyi, and Szpruch]{NSS19}
A.~Neuenkirch, M.~Sz\"olgyenyi, and L.~Szpruch.
\newblock An adaptive {E}uler-{M}aruyama scheme for stochastic differential equations with discontinuous drift and its convergence analysis.
\newblock \emph{SIAM Journal on Numerical Analysis}, 57\penalty0 (1):\penalty0 378--403, 2019.

\bibitem[Ngo and Taguchi(2016)]{ngo2016}
H.~L. Ngo and D.~Taguchi.
\newblock Strong rate of convergence for the {E}uler-{M}aruyama approximation of stochastic differential equations with irregular coefficients.
\newblock \emph{Mathematics of Computation}, 85\penalty0 (300):\penalty0 1793--1819, 2016.

\bibitem[Ngo and Taguchi(2017{\natexlab{a}})]{ngo2017a}
H.~L. Ngo and D.~Taguchi.
\newblock On the {E}uler-{M}aruyama approximation for one-dimensional stochastic differential equations with irregular coefficients.
\newblock \emph{IMA Journal of Numerical Analysis}, 37\penalty0 (4):\penalty0 1864--1883, 2017{\natexlab{a}}.

\bibitem[Ngo and Taguchi(2017{\natexlab{b}})]{ngo2017b}
H.~L. Ngo and D.~Taguchi.
\newblock Strong convergence for the {E}uler-{M}aruyama approximation of stochastic differential equations with discontinuous coefficients.
\newblock \emph{Statistics \& Probability Letters}, 125:\penalty0 55--63, 2017{\natexlab{b}}.

\bibitem[Przyby{\l}owicz(2016)]{Przybylowicz2016}
P.~Przyby{\l}owicz.
\newblock Optimal global approximation of stochastic differential equations with additive {P}oisson noise.
\newblock \emph{Numerical Algorithms}, 73\penalty0 (2):\penalty0 323--348, 2016.

\bibitem[Przyby{\l}owicz(2019{\natexlab{a}})]{Przybylowicz2019a}
P.~Przyby{\l}owicz.
\newblock Optimal sampling design for global approximation of jump diffusion stochastic differential equations.
\newblock \emph{Stochastics}, 91\penalty0 (2):\penalty0 235--264, 2019{\natexlab{a}}.

\bibitem[Przyby{\l}owicz(2019{\natexlab{b}})]{Przybylowicz2019b}
P.~Przyby{\l}owicz.
\newblock Efficient approximate solution of jump--diffusion {SDE}s via path-dependent adaptive step-size control.
\newblock \emph{Journal of Computational and Applied Mathematics}, 350:\penalty0 396--411, 2019{\natexlab{b}}.

\bibitem[Przyby{\l}owicz and Sz\"olgyenyi(2021)]{PS20}
P.~Przyby{\l}owicz and M.~Sz\"olgyenyi.
\newblock Existence, uniqueness, and approximation of solutions of jump-diffusion sdes with discontinuous drift.
\newblock \emph{Applied Mathematics and Computation}, 403:\penalty0 126191, 2021.

\bibitem[Przyby{\l}owicz et~al.(2021)Przyby{\l}owicz, Sz\"olgyenyi, and Xu]{PSX21}
P.~Przyby{\l}owicz, M.~Sz\"olgyenyi, and F.~Xu.
\newblock Existence and uniqueness of solutions of {SDE}s with discontinuous drift and finite activity jumps.
\newblock \emph{Statistics \& Probability Letters}, 174\penalty0 (109072), 2021.

\bibitem[Przyby{\l}owicz et~al.(2022{\natexlab{a}})Przyby{\l}owicz, Schwarz, Steinicke, and Sz{\"o}lgyenyi]{PSSS22}
P.~Przyby{\l}owicz, V.~Schwarz, A.~Steinicke, and M.~Sz{\"o}lgyenyi.
\newblock A {S}korohod measurable universal functional representation of solutions to semimartingale {SDE}s.
\newblock \emph{arXiv:2201.06278v2}, 2022{\natexlab{a}}.

\bibitem[Przyby{\l}owicz et~al.(2022{\natexlab{b}})Przyby{\l}owicz, Schwarz, and Sz{\"o}lgyenyi]{PSS22}
P.~Przyby{\l}owicz, V.~Schwarz, and M.~Sz{\"o}lgyenyi.
\newblock A higher order approximation method for jump-diffusion {SDE}s with discontinuous drift coefficient.
\newblock \emph{arXiv:2211.08739}, 2022{\natexlab{b}}.

\bibitem[Przyby{\l}owicz et~al.(2022{\natexlab{c}})Przyby{\l}owicz, Sobieraj, and St\c{e}pie\'{n}]{Przybylowicz2022}
P.~Przyby{\l}owicz, M.~Sobieraj, and {\L}.~St\c{e}pie\'{n}.
\newblock Efficient approximation of {SDE}s driven by countably dimensional {W}iener process and {P}oisson random measure.
\newblock \emph{SIAM Journal on Numerical Analysis}, 60\penalty0 (2):\penalty0 824--855, 2022{\natexlab{c}}.

\bibitem[Spendier and Sz{\"o}lgyenyi(2022)]{Spendier2022}
K.~Spendier and M.~Sz{\"o}lgyenyi.
\newblock Convergence of the tamed-{E}uler-{M}aruyama method for {SDE}s with discontinuous and polynomially growing drift.
\newblock \emph{arXiv:2212.08839}, 2022.

\bibitem[Yaroslavtseva(2022)]{Yaroslavtseva2022}
L.~Yaroslavtseva.
\newblock An adaptive strong order 1 method for {SDE}s with discontinuous drift coefficient.
\newblock \emph{Journal of Mathematical Analysis and Applications}, 513\penalty0 (2):\penalty0 126180, 2022.

\end{thebibliography}


}
\end{document}